\documentclass[11pt, reqno]{amsart}
\usepackage{amsmath,amssymb,amscd,amsthm}

\allowdisplaybreaks[2]

\newtheorem{theorem}{Theorem}[section]
\newtheorem{lemma}[theorem]{Lemma}

\theoremstyle{definition}

\newtheorem{remark}[theorem]{Remark}
\newtheorem*{ackno}{Acknowledgments}

\let\Re=\undefined\DeclareMathOperator*{\Re}{Re}
\let\Im=\undefined\DeclareMathOperator*{\Im}{Im}

\newcommand{\I}{\hspace{0.5mm}\text{I}\hspace{0.5mm}}
\newcommand{\II}{\text{I \hspace{-2.8mm} I} }

\newcommand{\R}{\mathbb{R}}
\newcommand{\C}{\mathbb{C}}

\newcommand{\loc}{\textup{loc}}

\newcommand{\RR}{\mathcal{R}}

\newcommand{\al}{\alpha}
\newcommand{\dl}{\delta}
\newcommand{\Dl}{\Delta}
\newcommand{\eps}{\varepsilon}

\newcommand{\g}{\gamma}

\newcommand{\ld}{\lambda}

\newcommand{\dt}{\partial_t}

\newcommand{\nb}{\nabla}

\numberwithin{equation}{section}
\numberwithin{theorem}{section}

\begin{document}
\title[Energy-critical NLS with non-vanishing boundary conditions]{Global well-posedness of the Gross--Pitaevskii and cubic-quintic
nonlinear Schr\"odinger equations with non-vanishing boundary conditions}

\author{Rowan Killip, Tadahiro Oh, Oana Pocovnicu, and Monica Vi\c{s}an}

\address
{Rowan Killip\\
Department of Mathematics\\
University of California, Los Angeles\\
Math Sciences Building 6167\\
Los Angeles, CA 90095, USA}
\email{killip@math.ucla.edu}

\address{
Tadahiro Oh\\
Department of Mathematics\\
Princeton University\\
Fine Hall\\
Washington Rd.\\
Princeton, NJ 08544-1000, USA}
\email{hirooh@math.princeton.edu}

\address{
Oana Pocovnicu\\
Department of Mathematics\\
Imperial College London\\
South Kensington Campus\\
London, SW7 2AZ, United Kingdom}

\email{o.pocovnicu@imperial.ac.uk}

\address
{Monica Vi\c{s}an\\
Department of Mathematics\\
University of California, Los Angeles\\
Math Sciences Building 6167\\
Los Angeles, CA 90095, USA}
\email{visan@math.ucla.edu}

\begin{abstract}
We consider the Gross--Pitaevskii equation on $\R^4$ and the cubic-quintic nonlinear Schr\"odinger equation (NLS)
on $\R^3$ with non-vanishing boundary conditions at spatial infinity. By viewing these equations as
perturbations to the energy-critical NLS,  we prove that they are globally well-posed in their energy spaces.
In particular, we prove unconditional uniqueness in the energy spaces for these equations.
\end{abstract}



\maketitle

\section{Introduction}

In this paper, we consider two models of NLS type with non-vanishing boundary conditions at spatial infinity, namely, the Gross--Pitaevskii equation \eqref{GP1}
in four spatial dimensions and the cubic-quintic model \eqref{CQNLS1} in three spatial dimensions.  Both models admit stable constant solutions; in this paper
we study excitations around these.  This constitutes the origin of the non-zero boundary conditions at spatial infinity.  The common theme of the two models
considered in this paper is the fact that the highest power appearing in the nonlinearity is energy-critical.

We start by discussing the Cauchy problem for the Gross--Pitaevskii equation:
\begin{equation} \label{GP1}
\begin{cases}
i\partial_tu+\Delta u=(|u|^2-1)u, & \quad ( t, x)  \in \R\times \R^n, \\
u\big|_{t = 0}=u_0,
\end{cases}
\end{equation}
with the non-vanishing boundary condition:
\begin{equation} \label{BC1}
\lim_{|x|\to \infty} |u(x)| = 1.
\end{equation}

The Gross--Pitaevskii equation \eqref{GP1} appears in various physical problems
\cite{Frisch, Gross, Pitaevskii, SULEM} such as  superfluidity of Helium II,
Bose--Einstein condensation, and  nonlinear optics (``dark" and ``black" solitons, optical vortices).
It constitutes the Hamiltonian evolution corresponding to the Ginzburg--Landau energy:
\begin{equation}\label{HamilGP0}
E(u)=\frac{1}{2}\int_{\R^n}|\nabla u|^2 \,dx + \frac{1}{4}\int_{\R^n}(|u|^2-1)^2\, dx.
\end{equation}

The Cauchy problem for the Gross--Pitaevskii equation \eqref{GP1} (with \eqref{BC1}) was first solved on $\R$ by Zhidkov \cite{Zhidkov} in what are now termed
Zhidkov spaces:
\[
X^k(\R^n):=\{u\in L^{\infty}(\R^n) : \partial^{\alpha}u\in L^2(\R^n), 1\leq |\alpha|\leq k\}.
\]
In spatial dimensions $n=2,3$, B\'ethuel and Saut \cite{Bethuel-Saut} proved global well-posedness in $1+H^1(\R^n)$ and
Gallo \cite{Gallo} proved it in Zhidkov spaces.  Finally, G\'erard \cite{PG1, PG3} proved global well-posedness on $\R^n$, $n=1,2,3$, in the energy space,
that is, the space of functions of finite Ginzburg--Landau energy.  He also proved that when $n = 3, 4$,  the energy space is given by
\begin{equation}
\label{EnergyGP0}
\mathcal{E}_\text{GP}(\R^n) =\Big\{u=\al+v:  |\al|=1, v\in \dot{H}^1(\R^n), |v|^2+2\Re(\bar{\al}v)\in L^2(\R^n)\Big\}.
\end{equation}
Moreover, if $u_0\in \mathcal{E}_{\text{GP}}$, then the corresponding solution $u(t)$ belongs to $\mathcal{E}_\text{GP}$ for all $t\in\R$.
Namely, we have $u(t)=\al+v(t)\in \mathcal{E}_\text{GP}$, where $v(t)$ has the same properties as $v_0$ and satisfies
\begin{equation} \label{GP2a}
\begin{cases}
i\partial_tv+\Delta v=\big(|v|^2+2\Re(\bar{\al}v)\big)(v+\al),\\
v\big|_{t = 0} = u_0-\al.
\end{cases}
\end{equation}
Since $|\al| = 1$, we can write $\al = e^{i\theta}$.  Then, by the gauge invariance of the equation under
$u \mapsto e^{-i\theta} u$,  we can assume that $\theta = 0$ and thus $\al = 1$.  Therefore, $u=1+v$ and $v$ satisfies
\begin{equation}\label{GP2}
\begin{cases}
i\partial_tv+\Delta v
=|v|^2v+2\Re(v)v+  |v|^2 + 2\Re(v), \\
v\big|_{t = 0} =v_0:=u_0-1.
\end{cases}
\end{equation}

With $u=1+v\in \mathcal{E}_\text{GP}(\R^4)$, we have $v\in \dot{H}^1(\R^4)$.  Then, using the Sobolev embedding $\dot{H}^1(\R^4)\subset L^4(\R^4)$ and \eqref{EnergyGP0},
we obtain $\Re(v)\in L^2(\R^4)$, and therefore $v\in H^1_{\text{real}}(\R^4)+i\dot{H}^1_{\text{real}}(\R^4)$. Here, we denoted by
$H^1_\text{real}(\R^4) := H^4(\R^4; \R)$ the Sobolev space of real-valued functions. Thus, in four dimensions the energy space is given by
\begin{equation}
\mathcal{E}_{\text{GP}}(\R^4)  =
\big \{ u=1+v: v \in H_\text{real}^1 (\R^4)+ i \dot H_\text{real}^1(\R^4)\big\}.
\label{EnergyGP}
\end{equation}
In this formulation, the Hamiltonian takes the form
\begin{align}
E(u)=E(1+v) &  = \frac{1}{2}\int_{\R^n} |\nb v|^2 \,dx +\frac{1}{4} \int_{\R^n} \big(|v|^2 + 2 \Re(v)\big)^2 \,dx.
\label{HamilGP}
\end{align}
To simplify notation, we will simply write $E(v)$ instead of $E(1+v)$.

Our main interest is to study the global behavior of solutions to the Gross--Pitaevskii equation \eqref{GP1} on $\R^4$.
In this paper,  we  prove global well-posedness in the energy space.  While of less importance from a physical point of view,
the four dimensional case is very interesting from an analytical point of view.  This is due to the fact that in \eqref{GP2}, the nonlinearity is
the sum of a cubic term and quadratic terms, and the cubic term is energy-critical ($\dot{H}^1_x$-critical) in four space dimensions.

Previously, G\'erard \cite{PG1} constructed global solutions to the Gross--Pitaevskii equation on $\R^4$ for initial data in the energy space
with {\it small} energy.  His proof is based on contraction mapping arguments in Strichartz spaces; correspondingly, the uniqueness statement is for solutions
in the energy space with $\nabla u\in L_{t, \loc}^2 L_x^4$. As the cubic nonlinearity is energy-critical, contraction mapping arguments cannot be
used to prove global existence of solutions with {\it large} energy, even if $v_0\in C^\infty_c(\R^4)$.

Let us now describe briefly what is known about the long-time behavior of solutions to \eqref{GP1}.  Writing $\tilde{u}:=e^{-it}u$, we obtain a solution to the
defocusing cubic nonlinear Schr\"odinger equation
\begin{equation*}
i\partial_t\tilde{u}+\Delta \tilde{u}=|\tilde{u}|^2\tilde{u}
\end{equation*}
with boundary condition $\lim_{|x|\to \infty} |\tilde{u}(x)| = 1$.  The non-vanishing boundary condition makes the long-time dynamics of
the Gross--Pitaevskii equation more complex than that of the usual NLS with zero boundary condition.
For instance, on $\R^n$ with $n=3,4$, solutions to the Gross--Pitaevskii equation with small energy scatter \cite{GNT3, GNT4},
while at higher levels of energy, one can build travelling wave solutions \cite{Bethuel-Saut, BGS, Maris_Existence of TW}. The dynamics of the
usual defocusing cubic NLS on $\R^n$ with $n=3,4$ is more straightforward: all solutions are known to scatter.  In this paper, we do not address the long-time behavior
of solutions to \eqref{GP1}.

Our main result about the Gross--Pitaevskii equation is as follows:

\begin{theorem} \label{THM:GP}
The Gross--Pitaevskii equation \eqref{GP1} is globally well-posed in the energy space $\mathcal{E}_\textup{GP}(\R^4)$.
In particular, solutions are unconditionally unique in $C_t(\R; \mathcal{E}_\textup{GP}(\R^4))$.
\end{theorem}

\begin{remark}
As the proof of Theorem~\ref{THM:GP} yields spacetime bounds on any compact time interval, our arguments also resolve the global regularity question for \eqref{GP1},
namely, that Schwartz excitations lead to global Schwartz solutions.
\end{remark}

The proof of Theorem \ref{THM:GP} is presented in Section \ref{SEC:GP}.  The idea of the proof is to treat \eqref{GP2} as the energy-critical cubic NLS on $\R^4$
with a subcritical  perturbation, and then to use  the perturbative approach developed by Tao, Vi\c{s}an, and Zhang \cite{TVZ}.
This allows us to construct global solutions to \eqref{GP1} for all initial data in the energy space.  These solutions satisfy local-in-time (Strichartz) spacetime bounds.
Finally, we prove unconditional uniqueness for these solutions, that is, they are unique in the larger class of solutions that are merely continuous (in time) with values in the energy space.

\medskip
The second topic of this paper is the Cauchy problem for the cubic-quintic NLS with non-vanishing boundary condition:
\begin{equation}\label{CQNLS1}
\begin{cases}
i \dt u + \Dl u = \al_1 u - \al_3 |u|^2 u +\al_5 |u|^4 u, & (t, x) \in \R \times \R^3, \\
u\big|_{t = 0} = u_0,
\end{cases}
\end{equation}
where $\al_1, \al_3, \al_5 > 0$ with $\al_3^2 - 4\al_1 \al_5 >0$.  This condition guarantees that  the polynomial
$\al_1 - \al_3 s + \al_5 s^2$ has two distinct real positive roots $r_0^2 > r_1^2 > 0$.  The boundary condition is given by
\begin{equation} \label{BCCQ1}
\lim_{|x|\to \infty} |u(x)| = r_0.
\end{equation}
The choice of the larger root guarantees energetic stability of the constant solution; this constitutes a local minimum of the energy functional \eqref{HamilCQ1} below.

The equation \eqref{CQNLS1} appears in a great variety of physical problems.  It is a model in superfluidity
\cite{Ginsburg1958, Ginsburg1976}, descriptions of bosons \cite{Barashenkov} and  of defectons \cite{Pushakarov1978},
the theory of ferromagnetic and molecular chains \cite{Pushakarov1984, Pushakarov1986}, and in nuclear hydrodynamics \cite{Kartavenko}.

By rescaling (both spacetime and the values of $u$), it suffices to consider the case $r_0^2 = 1$ and $\alpha_5=1$.  Thus, \eqref{CQNLS1} reduces to
\begin{equation}\label{CQNLS2}
\begin{cases}
i \dt u + \Dl u = (|u|^2 - 1) (|u|^2 - r_1^2 ) u,\\
u\big|_{t = 0} = u_0,
\end{cases}
\end{equation}
with the boundary condition:
\begin{equation} \label{BCCQ2}
\lim_{|x|\to \infty} |u(x)| = 1.
\end{equation}
Moreover, the Hamiltonian for \eqref{CQNLS2} is given by
\begin{align}
E(u) = \frac{1}{2} \int_{\R^3} |\nb u|^2 \, dx + \frac{\g}{4}  \int_{\R^3} (|u|^2-1)^2 \, dx + \frac{1}{6} \int_{\R^3} (|u|^2-1)^3 \, dx,
\label{HamilCQ1}
\end{align}
where $\g = 1-r_1^2> 0$.  As in the case of the Gross--Pitaevskii equation, we can write our solution as $u(t) = 1 + v(t)$ and so obtain
\begin{align}\label{CQNLS3}
\begin{cases}
i \dt v + \Dl v  =  |v|^4 v + \RR(v) , & (t, x) \in \R\times \R^3, \\
v\big|_{t = 0}  = v_0.
\end{cases}
\end{align}
Here, $\RR(v)$ represents the deviation from the energy-critical NLS and is given by
\begin{align} \label{CQNLS4}
\RR(v)  & =  |v|^4 + 4|v|^2 v \Re(v)\notag
+  4 |v|^2 \Re(v) + \g |v|^2 v + 4 v \Re(v)^2\notag \\
& \hphantom{XXXXXXXXX}
 +  \g |v|^2 + 4 \Re(v)^2 + 2\g v \Re(v) + 2\g \Re(v).
\end{align}
In this formulation, the Hamiltonian can be written as
\begin{align}
E(u)=E(1+v) = \frac{1}{2}& \int_{\R^3} |\nb v|^2 \, dx \notag \\
&  + \frac{\g}{4}  \int_{\R^3} \big(|v|^2+2\Re(v)\big)^2 \, dx  + \frac{1}{6} \int_{\R^3} \big(|v|^2+2\Re(v)\big)^3 \, dx.
\label{HamilCQ2}
\end{align}
As before,  we will simply write $E(v)$ for  $E(1+v)$.

Unlike the Hamiltonian \eqref{HamilGP} for the Gross--Pitaevskii equation, the Hamiltonian for the cubic-quintic NLS is not sign-definite
(at least when $\gamma<2/3$). In this paper, we define the ``energy space'' as
\begin{equation}
\mathcal{E}_{CQ}(\R^3):= 1+\big( H_\text{real}^1 (\R^3)+ i \dot H_\text{real}^1(\R^3)\big)\cap L^4(\R^3).\label{EnergyCQ0}
\end{equation}
In Subsection \ref{SUBSEC:CQ1}, we show that $u=1+v$ belongs to this energy space $\mathcal{E}_\text{CQ}(\R^3)$ if and only if $|E(u)|<\infty$ and $\Re(v) \in L^2(\R^3)$. This
allows us to prove global well-posedness of equation \eqref{CQNLS2} in $\big( H_\text{real}^1 (\R^3)+ i \dot H_\text{real}^1(\R^3)\big)\cap L^4(\R^3)$. As a consequence, we obtain
the second main result of this paper.

\begin{theorem} \label{THM:CQ}
The cubic-quintic NLS  \eqref{CQNLS2} $($and hence \eqref{CQNLS1} with $\al_3^2 - 4 \al_1 \al_5 > 0)$
is globally well-posed in the energy space $\mathcal{E}_\textup{CQ}(\R^3)$.
In particular, solutions are unconditionally unique in $C_t(\R; \mathcal{E}_\textup{CQ}(\R^3))$.
\end{theorem}

The proof of Theorem~\ref{THM:CQ} is presented in Section~\ref{SEC:CQ}.  As for Theorem \ref{THM:GP}, the main ingredient in the proof of Theorem \ref{THM:CQ}
is the perturbative approach developed in \cite{TVZ}. In this case, the Hamiltonian $E(v)$ in \eqref{HamilCQ2} does {\it not} control the $\dot{H}^1$-norm of $v$.
Nonetheless, it turns out that the quantity $M(v) : = E(v) + C_0 \int \lvert \Re(v)\rvert^2\, dx $, for some $C_0=C_0(\gamma) > 0$, controls the $\dot{H}^1$-norm of $v$.
While $E(v)$ is conserved under \eqref{CQNLS3}, this new quantity $M(v)$ is no longer conserved.  However, we show that $M(v(t))$ remains finite over
any finite time interval for any initial data $u_0$ in the energy class $\mathcal{E}_\text{CQ}(\R^3)$.  This is sufficient to construct global solutions for such initial
data.  Finally, we establish uniqueness of solutions that are merely continuous in the energy space.

\section{Notations and Perturbation Lemma}

\subsection{Notations}
We often use the notation $X\lesssim Y$ or $X=O(Y)$ whenever there exists some constant $C>0$ so that $X \leq CY$.  We use $X\sim Y$ if
$X\lesssim Y \lesssim X$.  We use $X\ll Y$ if $X \leq c Y$ for some small constant $c>0$.
Note that the derivative operator $\nb$ acts only on the spatial variables.

We say that a pair of exponents $(q, r)$ is Schr\"odinger-admissible if $\frac{2}{q} + \frac{n}{r} = \frac{n}{2}$ with $2\leq q, r \leq \infty$
and $(q, r, n) \ne (2, \infty, 2)$.  Given a spacetime slab $I\times\R^n$, we define the  $\dot{S}^0(I\times\R^n)$-Strichartz norm by
\begin{align*}
\|v\|_{\dot{S}^0(I)} = \|v\|_{\dot{S}^0(I\times\R^n)}:=&\sup \|v\|_{L_t^qL_x^r(I\times \R^n)},
\end{align*}
where the supremum is taken over all admissible pairs $(q, r)$. We define the  $\dot{S}^1(I\times\R^n)$-Strichartz norm by
\begin{align*}
\|v\|_{\dot{S}^1(I)} = \|v\|_{\dot{S}^1(I\times\R^n)}:=&\|\nabla v\|_{\dot{S}^0(I\times\R^n)}.
\end{align*}
We use $\dot{N}^0(I\times \R^n)$ to denote the dual space of $\dot{S}^0(I\times\R^n)$, and
\[\dot{N}^1(I\times\R^n) := \{ u:I\times\R^n\to \C: \nb u \in N^0(I \times \R^n)\}.\]

Next, we recall the Strichartz estimates; see Ginibre and Velo \cite{GV}, Keel and Tao \cite{KeelTao}, Strichartz \cite{Strichartz}, and Yajima \cite{Yajima}.

\begin{lemma}
Let $I$ be an interval in $\R$.  Then, for $j = 0, 1$  we have the following homogeneous Strichartz estimate:
\[\big\| e^{it\Dl} u_0\big\|_{\dot S^j(I\times \R^n)} \lesssim \|u_0\|_{\dot H^j_x(\R^n)}\]
and the inhomogeneous Strichartz estimate:
\[ \bigg\|\int_{t_0}^t e^{i(t-t')\Dl} F(t') \,dt' \bigg\|_{\dot S^j(I\times \R^n)} \lesssim \|F\|_{\dot N^j(I\times \R^n)}.\]
\end{lemma}

Some of the Schr\"odinger-admissible pairs for $n = 3, 4$ that we will use below are:
\begin{align*}
\bullet\  n = 4:&  \ (2, 4), \ (6, \tfrac{12}{5}),  \ (\infty, 2),\\
\bullet \ n = 3:& \ (2, 6),  \ (\tfrac{8}{3}, 4), \ (\tfrac{20}{7}, \tfrac{30}{8}), \ (5, \tfrac{30}{11}), \ (10, \tfrac{30}{13}), \  (20, \tfrac{15}{7}),  \ (\infty, 2).
\end{align*}

Lastly, given an interval $I$, we use $\dot{X}^1(I)= \dot{X}^1 (I\times \R^n)$
to denote the following spaces:
\begin{align*}
\bullet \ n = 4: \ \|v\|_{\dot{X}^1(I)}:=&\|\nabla v\|_{L^6_tL^{\frac{12}{5}}_x(I\times \R^4)},\\
\bullet \ n = 3: \ \|v\|_{\dot{X}^1(I)}:=&\|\nabla v\|_{L^{10}_tL^{\frac{30}{13}}_x(I\times \R^3)}.
\end{align*}

\subsection{Perturbation lemma}

In \cite{TVZ}, Tao, Vi\c{s}an, and Zhang considered global well-posedness and scattering questions for
NLS with combined power-type nonlinearities:
\begin{equation}
 i \dt u + \Dl u = \ld_1 |u|^{p_1} u + \ld_2 |u|^{p_2} u,
\qquad 0 < p_1 < p_2 \leq \tfrac{4}{n-2} \quad \text{and} \quad \lambda_1, \lambda_2=\pm 1.
\label{NLS0}
\end{equation}
In particular, when $p_2 = \frac{4}{n-2}$ and $\lambda_2=1$, they treated \eqref{NLS0} as a perturbation of the defocusing energy-critical NLS:
\begin{equation} \label{NLS1}
i\partial_tw+\Delta w= |w|^{\frac{4}{n-2}}w.
\end{equation}

Global well-posedness and scattering for \eqref{NLS1} was proved by Colliander, Keel, Staffilani, Takaoka, and Tao \cite{CKSTT} for
$n = 3$ and subsequently by Ryckman and Vi\c{s}an \cite{RV} and Vi\c{s}an \cite{Visan:Duke} for $n\geq 4$.  The three and four dimensional
results (which underpin this paper) were revisited in \cite{KV:gopher} and \cite{V11} in light of recent developments.
Importantly for the considerations of this paper, all these works also proved global spacetime bounds.  Specifically, if $w$ is a solution to
the energy-critical NLS \eqref{NLS1} with initial data $w_0 \in \dot{H}^1(\R^n)$, then
\begin{align}\label{W1}
\| w\|_{\dot{S}^1(\R\times \R^n)}&\leq C(\|w_0\|_{\dot{H}^1}).
\end{align}

To prove global well-posedness for solutions to \eqref{NLS0} (with $\lambda_1=\lambda_2=1$ and $p_2=\frac4{n-2}$), Tao, Vi\c{s}an, and Zhang combined
the spacetime bounds \eqref{W1} with a perturbation lemma \cite[Lemma 3.8]{TVZ}, which appeared in \cite{TV}.  In this spirit, we regard the Gross--Pitaevskii
equation \eqref{GP2} on $\R^4$ and the cubic-quintic NLS \eqref{CQNLS3} on $\R^3$ as perturbations to \eqref{NLS1} on $\R^4$ and $\R^3$, respectively.
Hence, our basic strategy is also to use a perturbative approach.

The key perturbation result we will use is Lemma~\ref{LEM:Perturb} below, which is \cite[Theorem~3.8]{ClayNotes} and represents a strengthening of
the result in \cite{TV}.  The additional strength is not needed for our purposes here; however, the statement is simpler and it makes our arguments
slightly simpler, too.

\begin{lemma}[Perturbation lemma, \cite{ClayNotes}]\label{LEM:Perturb}
Let $n\geq 3$ and let $I$ be a compact time interval.  Let $\tilde{w}$ be a solution on $I\times \R^n$ to the perturbed equation:
\begin{equation}\label{NLS2}
i\partial_t\tilde{w}+\Delta \tilde{w}=|\tilde{w}|^{\frac{4}{n-2}}\tilde{w} +e
\end{equation}
for some function $e$.  Suppose that there exist $L, E_0, E' > 0$ and some $t_0\in I$ such that
\begin{align}
\label{P1}
\|\tilde{w}\|_{L^{\frac{2(n+2)}{n-2}}_{t, x} (I\times \R^n)} & \leq L, \\
\label{P2}
\|\tilde{w}\|_{L^{\infty}_t \dot{H}^1_x(I\times \R^n)}  & \leq E_0, \\
\label{P3}
\|\tilde{w}(t_0)-w(t_0)\|_{\dot{H}^1(\R^n)}& \leq E'.
\end{align}
Furthermore, assume that we have
\begin{gather}\label{P4}
\bigl\|e^{i(t-t_0)\Delta}(\tilde{w}(t_0)-w(t_0))\big\|_{L^{\frac{2(n+2)}{n-2}}_{t,x}(I\times \R^n)}\leq\eps,\\
\label{P5}
\|\nabla e\|_{\dot{N}^0(I\times\R^n)}\leq\eps,
\end{gather}
for some $0<\eps\leq\eps_0$, where $\eps_0=\eps_0(E_0,E',L)>0$ is a small constant.  Then, there exists a solution $w$ to \eqref{NLS1} on $I\times \R^n$
with the specified initial data $w(t_0)$ at time $t = t_0$ satisfying
\begin{align}
\|w-\tilde{w}\|_{L^{\frac{2(n+2)}{n-2}}_{t,x} (I\times \R^n)}& \leq  C(E_0,E', L)\eps^c \label{P6}\\
\|w-\tilde{w}\|_{\dot{S}^1(I\times\R^n)}& \leq C(E_0,E', L)E'\label{P7}\\
\|w\|_{\dot{S}^1(I\times\R^n)} & \leq  C(E_0,E',L), \label{P8}
\end{align}
where $C(E_0, E', L) > 0$ is a non-decreasing function of $E_0$, $E'$, and $L$ and $c=c(n)>0$.
\end{lemma}

\begin{remark}\label{REM:Perturb}\rm
By the Strichartz inequality together with Sobolev embedding, condition \eqref{P4} is redundant if $E' = O(\eps)$.
\end{remark}

We will apply this lemma with $\tilde w=v$, thereby showing that it differs little from a solution $w$ to \eqref{NLS1}.
This will allow us to transfer the known spacetime bounds \eqref{W1} for $w$ to $v$, at least locally in time.  Note that in
the case of the Gross--Pitaevskii equation, condition \eqref{P2} is guaranteed by the conservation of energy \eqref{HamilGP}.
In the cubic-quintic case, the failure of the energy \eqref{HamilCQ2} to be coercive adds an additional wrinkle to the analysis.


\section{Global Well-Posedness of the Gross--Pitaevskii Equation on $\R^4$} \label{SEC:GP}

In this section, we present the proof of Theorem \ref{THM:GP}.  This breaks naturally into two parts: global existence and unconditional uniqueness.

\subsection{Global existence}
By time reversibility of the equation \eqref{GP2},
we need only consider the problem for $t\geq 0$.  Let $v_0$ be such that  $1+v_0 \in \mathcal{E}_\text{GP}(\R^4)$.  By the usual local well-posedness arguments,
it suffices to prove that there exists $T = T(E(v_0)) > 0$ such that
\begin{equation}\label{GPP2}
\|v\|_{\dot{S}^1([0, T]\times \R^4)} \leq C(E(v_0)),
\end{equation}
that is, we assume the solution $v$ exists on $[0,T]$ and prove that it satisfies \eqref{GPP2}.  Indeed, it is not difficult to prove local existence
via contraction mapping in the space $\dot X^1$ by using the estimates below (\eqref{GPP3c} in particular).  Notice that this argument gives local existence on any
sufficiently short interval $I$ for which $\|e^{it\Delta}v_0\|_{\dot X^1(I)}$ is sufficiently small.  Strichartz estimates together with the monotone/dominated convergence
theorems guarantee that one can always find such a short interval $I$.

Finally, combining \eqref{GPP2} with conservation of the Hamiltonian $E(v)$, one can iterate the local argument, thus obtaining a global solution $v$ to \eqref{GP2} that is
continuous (in time) with values in the energy space and lies in $\dot S^1 (I)$ for any compact time interval $I$.

As mentioned before, the main idea is to view \eqref{GP2} as a perturbation to the energy-critical cubic NLS \eqref{NLS1}, that is, regard $v$ as $\tilde{w}$ in Lemma
\ref{LEM:Perturb} with \[e = 2\Re(v)v + |v|^2 +2\Re(v).\] The argument follows closely the one in Section 4 of \cite{TVZ}.

Let $w$ be the global solution to the energy-critical cubic NLS \eqref{NLS1} with initial data $v_0$.
Divide the interval $\R_+ = [0, \infty)$ into $J=J(E(v_0) ,\eta)$ many subintervals $I_j=[t_j,t_{j+1}]$ such that
\begin{equation} \label{GPP1}
\| w \|_{\dot{X}^1(I_j)} = \|\nabla w\|_{L^6_tL^{\frac{12}{5}}_x(I_j \times \R^4)}\sim \eta,
\end{equation}
for some small $\eta> 0$ to be chosen later.

Fix $T>0$ (to be chosen later in terms of $E(v_0)$), and write $[0, T] =\bigcup_{j=0}^{J'}\big([0,T]\cap I_j\big)$
for some $J' \leq J$, where $[0,T]\cap I_j\neq \emptyset$ for $0\leq j\leq J'$.

Since the nonlinear evolution $w$ is small on $I_j$, it follows that the linear evolution $e^{i(t-t_j)\Delta}w(t_j)$ is also small on $I_j$.
Indeed, using the Duhamel formula:
\begin{equation*}
w(t)=e^{i(t-t_j)\Delta}w(t_j)-i\int_{t_j}^t e^{i(t-s)\Delta}|w(s)|^2w(s)\, ds \quad \text{for} \quad t \in I_j,
\end{equation*}
together with Strichartz estimates, Sobolev embedding, and \eqref{GPP1}, we obtain
\begin{align*}
\big\| e^{i(t-t_j)\Dl} w(t_j)\big\|_{\dot{X}^1(I_j)}
&\leq \|w\|_{\dot{X}^1(I_j)}+C \|w^2\nabla w\|_{L^2_{t}L^\frac{4}{3}_x(I_j\times\R^4)}\\
& \leq \eta + C\|\nabla w\|_{L^6_{t} L^\frac{12}{5}_x (I_j \times\R^4)}\|w\|_{L^6_{t, x}(I_j\times \R^4)}^2\\
& \leq \eta + C\|\nabla w\|_{L^6_tL^\frac{12}{5}_x(I_j\times \R^4)}^3\\
& \leq \eta+C\eta^3,
\end{align*}
where $C$ is an absolute constant. Therefore, if $\eta$ is small enough, we obtain
\begin{align}\label{GPP3}
\big\| e^{i(t-t_j)\Dl} w(t_j)\big\|_{\dot{X}^1(I_j)}
\leq 2\eta.
\end{align}

Now, we estimate $v$ on the first interval  $I_0$. Arguing as before, using \eqref{GPP3} and $w(0) = v_0$, we obtain
\begin{align}
\|v\|_{\dot{X}^1(I_0)}
& \leq \| e^{it\Delta}v_0\|_{\dot{X}^1(I_0)}
+ C\|v^2\nabla v\|_{L^2_t L^\frac{4}{3}_x (I_0\times \R^4)} \notag \\
& \hphantom{XXX}
+C\|v\nabla v\|_{L^\frac{6}{5}_t L^\frac{12}{7}_x (I_0\times \R^4)}
+ C\|\nabla v\|_{L^1_t L^2_x (I_0\times \R^4)} \notag \\
& \leq \| e^{it\Delta}v_0\|_{\dot{X}^1(I_0)}
 +C\|\nabla v\|^3_{L^6_tL^{\frac{12}{5}}_x(I_0\times \R^4)} \notag \\
& \hphantom{XXX}
 +C|I_0|^\frac{1}{2}\| v\|_{L^6_{t, x} (I_0\times \R^4)}\|\nabla v\|_{L^6_tL^{\frac{12}{5}}_x(I_0 \times \R^4)}
+C|I_0|\|\nabla v\|_{L^\infty_tL^2_x(I_0 \times \R^4)} \notag \\
& \leq  2\eta +C\|v\|^3_{\dot{X}^1(I_0)} +CT^\frac{1}{2}\|v\|^2_{\dot{X}^1(I_0)}+ CTE(v_0)^{\frac{1}{2}}.
\label{GPP3c}
\end{align}
Choosing $\eta \ll 1$ and  $T\ll \min\{1,\eta E(v_0)^{-\frac{1}{2}}\}$, it follows from a standard continuity argument that
\begin{equation} \label{GPP4}
\|v\|_{\dot{X}^1(I_0)}\leq 3\eta.
\end{equation}
Thus, by Sobolev embedding, $\|v\|_{L^6_{t, x} (I_0 \times \R^4)} \leq C\eta$ and so condition \eqref{P1} in Lemma \ref{LEM:Perturb}
is satisfied with $L=C\eta$.  This completes the restrictions on the small absolute constant $\eta$.

From \eqref{HamilGP} and the conservation of the Hamiltonian, we have
\[
\|v\|_{L^{\infty}_t\dot{H}^1_x(I_0\times\R^4)}^2\leq 2E(v)=2E(v_0),
\]
and thus condition \eqref{P2} is satisfied with $E_0=[2E(v_0)]^{1/2}$. With $t_0 = 0$, conditions \eqref{P3} and \eqref{P4} are automatically satisfied
(with $E' = 0$ for \eqref{P3}).

Let us now verify condition \eqref{P5}.  Recall that  $e= 2v\Re(v)+|v|^2 + 2\Re( v)$. By the H\"older and Sobolev inequalities and \eqref{GPP4}, we have
\begin{align}
\|\nabla e\|_{\dot{N}^0(I_0)}
&\leq C\|v\nabla v\|_{L^\frac{6}{5}_t L^\frac{12}{7}_x( I_0\times \R^4)}+ C\|\nabla v\|_{L^1_tL^2_x(I_0\times \R^4)}\notag \\
&\leq C|I_0|^{\frac{1}{2}}
\|\nabla v\|^2_{L^6_tL^{\frac{12}{5}}_x(I_0\times \R^4)}
+ C|I_0| \|\nabla v\|_{L^{\infty}_tL^2_x(I_0\times \R^4)}\notag \\
&\leq CT^\frac{1}{2}\|v\|_{\dot{X}^1(I_0)}^2 + CTE(v_0)^{\frac{1}{2}}
\leq CT^\frac{1}{2}\eta^2 + CTE(v_0)^{\frac{1}{2}}.
\label{GPP4a}
\end{align}
Given $\eps>0$, we may choose $T = T(E(v_0), \eps)$ sufficiently small so that
$$
\|\nabla e\|_{\dot{N}^0(I_0)}\leq \eps.
$$
For $\eps\leq \eps_0$ with $\eps_0=\eps_0(E(v_0))$ dictated by Lemma~\ref{LEM:Perturb}, condition \eqref{P5} is thus satisfied.

Therefore, all hypotheses of Lemma~\ref{LEM:Perturb} are satisfied on the interval $I_0$,
provided $T = T(E(v_0),\eps)$ is chosen sufficiently small.  Hence we obtain
\begin{align}
\|w-v\|_{\dot{S}^1(I_0\times\R^4)} & \leq C\bigl(E(v_0)\bigr)\eps. \label{GPP5b}
\end{align}

Next, we consider the second interval $I_1$.  Condition \eqref{P3} is satisfied on $I_1$ with $E'=C(E(v_0))\eps$; indeed, \eqref{GPP5b} yields
\begin{align}\label{E grow}
\|w(t_1)-v(t_1)\|_{\dot{H}^1(\R^4)} \leq\|w-v\|_{\dot{S}^1(I_0\times\R^4)}\leq C\bigl(E(v_0)\bigr)\eps.
\end{align}
In view of Remark~\ref{REM:Perturb}, condition \eqref{P4} is satisfied for $\eps \leq \eps_0(E(v_0))$, where $\eps_0(E(v_0))$
is dictated by Lemma~\ref{LEM:Perturb}.  (Note that $E'$ grows with each successive interval $I_j$, but its ultimate size is $C(J',E(v_0))\eps$.)

By Strichartz with \eqref{E grow}, we have
\begin{align}\label{GPP6}
\big\|e^{i(t-t_1)\Delta}(v(t_1)-w(t_1))\big\|_{\dot{X}^1(I_1)}
\leq C\bigl(E(v_0)\bigr)\eps.
\end{align}
Then, proceeding as in \eqref{GPP3c} and using \eqref{GPP3} and \eqref{GPP6},
\begin{align*}
\|v\|_{\dot{X}^1(I_1)}
& \leq \big\|e^{i(t-t_1)\Delta}v(t_1)\big\|_{\dot{X}^1(I_1)} +C\|v\|_{\dot{X}^1(I_1)}^3
+CT^\frac{1}{2}\|v\|_{\dot{X}^1(I_1)}^2 +CTE(v_0)^{\frac{1}{2}} \\
& \leq \big\|e^{i(t-t_1)\Delta} w(t_1)\big\|_{\dot{X}^1(I_1)}
+\big\|e^{i(t-t_1)\Delta} (v(t_1)-w(t_1))\big\|_{\dot{X}^1(I_1)}\\
&\hphantom{XXX} +C\|v\|_{\dot{X}^1(I_1)}^3
+CT^\frac{1}{2}\|v\|_{\dot{X}^1(I_1)}^2 +CTE(v_0)^{\frac{1}{2}} \\
& \leq  2\eta+C\bigl(E(v_0)\bigr)\eps
+C\|v\|_{\dot{X}^1(I_1)}^3
+CT^\frac{1}{2}\|v\|_{\dot{X}^1(I_1)}^2 +CTE(v_0)^{\frac{1}{2}}.
\end{align*}
As $\eta\ll 1$, choosing $T = T(E(v_0), \eta)$ and $\eps = \eps(E(v_0), \eta) $ sufficiently small, we conclude that
\begin{align} \label{GPP7}
\|v\|_{\dot{X}^1(I_1)}\leq 3\eta
\end{align}
by a continuity argument.  Thus, by Sobolev embedding, we see that condition \eqref{P1} in Lemma \ref{LEM:Perturb}
is satisfied with $L=C\eta$ as before.  Condition \eqref{P2} is satisfied with $E_0=[2E(v_0)]^{1/2}$.  Lastly, as in \eqref{GPP4a}, using \eqref{GPP7}, we have
\begin{align*}
\|\nabla e\|_{\dot{N}^0(I_1)}
&\leq CT^\frac{1}{2}\|v\|_{\dot{X}^1(I_1)}^2 + CTE(v_0)^{\frac{1}{2}}
\leq C T^\frac{1}{2}\eta^2+ CTE(v_0)^{\frac{1}{2}}\leq \eps
\end{align*}
for $T = T(E(v_0), \eps)$ sufficiently small.  Therefore, condition \eqref{P5} is satisfied, provided $\eps \leq \eps_0(E(v_0))$,
where $\eps_0(E(v_0))$ is dictated by Lemma~\ref{LEM:Perturb}.  Hence, by Lemma \ref{LEM:Perturb} applied to the interval $I_1$,
\begin{align*}
\|v-w\|_{\dot{S}^1(I_1)}\leq C\bigl(E(v_0)\bigr)\eps.
\end{align*}

Arguing inductively, we obtain
\begin{align*}
\|v\|_{\dot{X}^1(I_j)}\leq 3\eta
\end{align*}
for all $0\leq j\leq J'$, provided that $\eps $ and $T$ are small, depending only on $E(v_0)$ (and $\eta\ll1$ and $J'$).
As $J'\leq J=J(E(v_0),\eta)$, we obtain $T=T(E(v_0))$ and
\begin{align}
\|v\|_{\dot{X}^1([0,T])}\leq 3\eta J(E(v_0),\eta)\leq C\bigl(E(v_0)\bigr).
\label{GPP8}
\end{align}
Using \eqref{GPP8}, we can estimate the $\dot{S}^1$-norm of the solution $v$ on $[0,T]$ by
\begin{align*}
\|v\|_{\dot{S}^1([0,T])}
&\lesssim \|v_0\|_{\dot{H}^1}+\|v\|_{\dot{X}^1([0,T])}^3
+T^{\frac{1}{2}}\|v\|^2_{\dot{X}^1([0,T])}+T\|\nabla v\|_{L^{\infty}_tL^2_x([0,T]\times \R^4)}\\
&\lesssim  E(v_0)+ C\bigl(E(v_0)\bigr)^3 + T^{\frac{1}{2}}C\bigl(E(v_0)\bigr)^2+TE(v_0)^{\frac{1}{2}}.
\end{align*}
As $T=T(E(v_0))$, we finally obtain that $\|v\|_{\dot{S}^1([0,T])}\leq C(E(v_0))$.
This allows us to iterate the local argument on time intervals of length $T = T(E(v_0)) >0$,
thus constructing a global solution $v$ with initial data $v_0$ such that $1+v_0$ lies in
the energy space $\mathcal{E}_\text{GP}(\R^4) =  H_\text{real}^1 (\R^4)+ i \dot H_\text{real}^1(\R^4)$.
While the  construction only yields  $v(t) \in \dot{H}^1(\R^4)$ for $t \in \R$, we can easily see that
$1+v(t) \in \mathcal{E}_\text{GP}(\R^4)$ for $t \in \R$ from the conservation of the Hamiltonian.

\subsection{Unconditional uniqueness}

We turn now to showing that the global solutions constructed above are unique among those that are continuous (in time) with values in the energy space.
We mimic the arguments in \cite[\S16]{CKSTT}.  To this end, let $v_0$ be such that $1+v_0\in \mathcal{E}_\text{GP}(\R^4)$ and let $v$ be the global
solution to \eqref{GP2} constructed above.  In particular, $v\in \dot S^1(I)$ for any compact time interval $I$.

Let $\tilde v:[0, \tau]\times\R^4\to \C$ be a second solution to \eqref{GP2} with the same initial data
such that $1+ \tilde v \in C_t([0, \tau];\mathcal{E}_\text{GP}(\R^4))$ and write
$\omega:=v-\tilde v$.  As $\omega(0)=0$ and $\omega$ is continuous in time, shrinking $\tau$ if necessary, we may assume
\begin{align}\label{e small}
\lVert\Re(\omega)\rVert_{L_t^\infty H^1_x([0, \tau]\times\R^4)} + \lVert\Im(\omega)\rVert_{L_t^\infty \dot H^1_x([0, \tau]\times\R^4)} \leq \eta
\end{align}
for a small $\eta>0$ to be chosen shortly.  By Sobolev embedding, this yields
\begin{align}\label{l4 small}
\|\omega\|_{L_t^\infty L^4_x([0, \tau]\times\R^4)}\lesssim \eta;
\end{align}
in particular, $\omega\in L^2_tL_x^4([0, \tau]\times\R^4)$.  Recalling that $v\in \dot S^1(I)$ for any compact time interval $I$ and further shrinking $\tau$ if necessary,
we may also assume that
\begin{align}\label{v small}
\|v\|_{L_{t,x}^6([0, \tau]\times\R^4)}\leq \eta.
\end{align}

Writing
\begin{align*}
\bigl[|v|^2v + 2\Re(v)v + |v|^2 & +2\Re(v)\bigr]- \bigl[|\tilde v|^2\tilde v + 2\Re(\tilde v)\tilde v + |\tilde v|^2 +2\Re(\tilde v)\bigr]\\
&= O\bigl(|\omega|^3 + |\omega||v|^2 + |\omega|^2 + |\omega||v| + |\Re(\omega)|\bigr)
\end{align*}
and using the Strichartz inequality together with \eqref{e small}, \eqref{l4 small}, \eqref{v small}, and H\"older, we estimate
\begin{align*}
\|\omega\|_{L_t^2L_x^4} &+ \lVert\Re(\omega)\rVert_{L_t^\infty L_x^2}\\
&\lesssim \|\omega^3\|_{L_t^2L_x^{\frac{4}{3}}} + \|\omega v^2\|_{L_t^{\frac{6}{5}}L_x^{\frac{12}{7}}} + \|\omega^2\|_{L_t^1L_x^2} +\|\omega v\|_{L_t^1L_x^2}
+\lVert\Re(\omega)\rVert_{L_t^1L_x^2}\\
&\lesssim \|\omega\|_{L_t^2L_x^4} \Bigl[\|\omega\|_{L_t^\infty L^4_x}^2
+ \|v\|_{L_{t,x}^6}^2 + \tau^{\frac{1}{2}} \|\omega\|_{L_t^\infty L^4_x}
+ \tau^{\frac{1}{2}} \|v\|_{L_t^\infty L^4_x}\Bigr] +\tau \lVert\Re(\omega)\rVert_{L_t^\infty L_x^2}\\
&\lesssim \bigl(\eta^2+\eta \tau^{\frac{1}{2}}+\tau^{\frac{1}{2}}\bigr)\|\omega\|_{L_t^2L_x^4} + \tau\lVert\Re(\omega)\rVert_{L_t^\infty L_x^2}.
\end{align*}
All spacetime norms in the display above are taken on $[0, \tau]\times\R^4$.  Taking $\eta$ sufficiently small and shrinking $\tau$ further if necessary, we obtain
$$
\|\omega\|_{L_t^2L_x^4([0, \tau]\times\R^4)} + \lVert\Re(\omega)\rVert_{L_t^\infty L_x^2([0, \tau]\times\R^4)}=0,
$$
which proves $v=\tilde v$ almost everywhere on $[0, \tau]\times\R^4$.

By time translation invariance, this argument can be applied to any sufficiently short time interval, which yields global unconditional uniqueness.
This completes the proof of Theorem~\ref{THM:GP}.


\section{Global Well-Posedness of the Cubic-Quintic NLS on $\R^3$}
\label{SEC:CQ}

\subsection{The energy space and energy estimates}
\label{SUBSEC:CQ1}

Let $E(v)$ be the Hamiltonian for \eqref{CQNLS3} given in \eqref{HamilCQ2}. It is not sign-definite in general and it does not seem to
control the $\dot{H}^1_x$-norm of $v$. However, once we add a multiple of the $L^2_x$-norm of $\Re(v)$, the following lemma shows
that it indeed controls the $\dot{H}^1_x$-norm of $v$.

\begin{lemma}\label{LEM:CQ1}
Let $E(v)$ be the Hamiltonian  as in \eqref{HamilCQ2}.  Then, there exists $C_0=C_0(\gamma)>0$ such that
\begin{equation}\label{CQbd1}
\int_{\R^3} |\nb v|^2 \,dx + \int_{\R^3} |v|^6 \,dx  + \gamma \int_{\R^3}  |v|^4 \,dx
\lesssim E(v) + C_0 \int_{\R^3} \lvert\Re(v)\rvert^2 \,dx.
\end{equation}
\end{lemma}

\begin{proof}
By writing $|v|^2 = \big(|v|^2 + 2\Re(v)\big) - 2\Re(v)$, we obtain
\begin{align}\label{v4}
\gamma |v|^4\leq  2\gamma\big(|v|^2 + 2\Re(v)\big)^2 + 8\gamma\lvert\Re(v)\rvert^2
\leq 8\Bigl[\tfrac\gamma 4 \big(|v|^2 + 2\Re(v)\big)^2 + \gamma\lvert\Re(v)\rvert^2\Bigl].
\end{align}

For  $a,b\in\R$ with $a+b\geq 0$, we have the inequality $(a+b)^3\leq 4a^3+4b^3$.  From this, we obtain
\begin{align*}
|v|^6=\big[|v|^2+2\Re(v)+(-2\Re(v))\big]^3\leq 4\big(|v|^2+2\Re(v)\big)^3+32\lvert\Re(v)\rvert^3.
\end{align*}
By Young's inequality,
\begin{align*}
\lvert\Re(v)\rvert^3
&\leq \frac{3}{4\delta^{\frac{4}{3}}}\lvert\Re(v)\rvert^2+\frac{\delta^4}{4}|v|^6,
\end{align*}
for any $\dl > 0$.  Combining these with $\delta=1/2$, we obtain
\begin{align}\label{v6}
|v|^6\leq 8 \big(|v|^2 + 2\Re(v)\big)^3 + 144\lvert\Re(v)\rvert^2 \leq 48 \Bigl[\tfrac16 \big(|v|^2 + 2\Re(v)\big)^3 + 3\lvert\Re(v)\rvert^2  \Bigr].
\end{align}

Collecting \eqref{v4} and \eqref{v6} and using the formula for the Hamiltonian \eqref{HamilCQ2}, we derive the claim.
\end{proof}

In view of this lemma, we can characterize the energy space $\mathcal{E}_{CQ}(\R^3)$ introduced in \eqref{EnergyCQ0} in terms of $E(v)$ and $\lVert\Re(v)\rVert_{L_x^2}$ only:

\begin{lemma}\label{LEM:CQ2}
Let $E(v)$ be   as in \eqref{HamilCQ2}.  Then $1+ v $ belongs to the energy space $\mathcal{E}_{CQ}(\R^3)$
if and only if $|E(v)| < \infty$ and $\textup {Re}(v) \in L^2(\R^3)$.
\end{lemma}

\begin{proof}
Using Sobolev embedding, it is easy to see that if $1+ v$ belongs to $\mathcal{E}_{CQ}(\R^3)$, then we have $|E(v)| <\infty$ and $\Re(v) \in L^2(\R^3)$. The converse follows from
Lemma \ref{LEM:CQ1}.
\end{proof}

While the Hamiltonian $E(v)$ is conserved under the flow of \eqref{CQNLS3}, the $L^2_x$-norm of $\Re(v)$ is not. The following lemma controls the growth of this quantity for
solutions in the energy space $\mathcal{E}_{CQ}(\R^3)$.

\begin{lemma} \label{LEM:CQ3}
Let $v$ be a solution to the cubic-quintic NLS \eqref{CQNLS3} on a time interval $[0, \tau]$ with initial data $v_0$ such that $1+v_0 \in \mathcal{E}_{CQ}(\R^3)$.
Then $M(v) := E(v) + C_0 \int_{\R^3} |\textup {Re}(v)|^2 \, dx$ $($with $C_0$ as in Lemma~\ref{LEM:CQ1}$)$ satisfies the following growth estimate:
\begin{equation} \label{CQB0}
M(v(t)) \leq M(v_0) e^{C_1 t}
\end{equation}
for all $t \in [0, \tau ]$ and some $C_1> 0$.

Moreover, there exists $C\big(E(v_0), \lVert\Re(v_0)\rVert_{L^2(\R^3)}, \tau\big) > 0$ such that
\begin{equation} \label{CQB0a}
\sup_{t\in[0,\tau]}\|v(t)\|_{\mathcal{E}_{CQ}(\R^3)} \leq C\big(E(v_0), \lVert\Re(v_0)\rVert_{L^2(\R^3)}, \tau\big),
\end{equation}
where
\[\|f\|_{\mathcal{E}_{CQ}(\R^3)}: = \lVert\Re(f)\rVert_{H^1(\R^3)} + \lVert\Im(f)\rVert_{\dot{H}^1(\R^3)} + \|f\|_{L^4(\R^3)}.\]
In particular, the $\dot{H}^1$-norm of the solution does not blow up in finite time.
\end{lemma}

\begin{proof}
We prove the lemma for smooth $v$.  From \eqref{CQNLS3}, we have
\begin{align}
\dt  \int_{\R^3} & \lvert\Re(v)\rvert^2 \, dx  = 2 \int_{\R^3} \Re(v)\,  \dt \Re(v) \,dx \notag \\
& = - 2 \int_{\R^3} \Re(v) \Im(\Dl v) \,dx + 2 \int_{\R^3} \Re(v) \Im \big(|v|^4v +  \RR(v)\big) \,dx
= : \I + \II. \label{CQB1}
\end{align}
Integrating by parts, we have
\begin{align}
|\I| = \bigg| 2 \int_{\R^3} \Re(\nb v) \cdot \Im(\nb v)\, dx  \bigg| \lesssim \int_{\R^3}  |\nb v |^2\, dx,
\label{CQB2}
\end{align}
while from \eqref{CQNLS4}, we have
\begin{align*}
\II & = 2 \int_{\R^3}
\Re(v)  |v|^4 \Im (v)
+ 4 \Re(v)^2 |v|^2 \Im(v)  \\
& \hphantom{XXXX}
+ \g \Re(v) |v|^2 \Im(v)
+ 4 \Re(v)^3 \Im(v)
+ 2\g \Re(v)^2 \Im(v) \,dx \\
& = : \II_1 + \II_2 + \II_3 + \II_4 +\II_5.
\end{align*}

Trivially, $ |\II_1| \lesssim \int_{\R^3} |v|^6\, dx .$  By interpolation and Young's inequality,
\begin{align}
|\II_2| \lesssim  \int_{\R^3}\lvert\Re(v)\rvert^4\, dx + \int_{\R^3} |v|^6\, dx
 \lesssim  \int_{\R^3}\lvert\Re(v)\rvert^2\, dx + \int_{\R^3} |v|^6\, dx.
\label{CQB3}
\end{align}
Next, by Cauchy's inequality ($ab \leq \frac{a^2}{2}+ \frac{b^2}{2}$ for $a, b \in \R$),   
\begin{align}
|\II_3| + |\II_4| \lesssim  \int_{\R^3}\lvert\Re(v)\rvert^2 \,dx + \int_{\R^3} |v|^6 \,dx
\label{CQB4}
\end{align}
and
\begin{align}
|\II_5| \lesssim  \int_{\R^3}\lvert\Re(v)\rvert^2\, dx  + \int_{\R^3}|v|^4\, dx.
\label{CQB5}
\end{align}

Collecting \eqref{CQB1} through \eqref{CQB5} and using conservation of energy and Lemma~\ref{LEM:CQ1}, we conclude that
\begin{align*}
\dt  M(v(t)) = C_0 \,\dt\! \int_{\R^3} & \lvert\Re(v)\rvert^2 \, dx \leq C_1 \Bigl[E(v) + C_0 \int_{\R^3} \lvert\Re(v)\rvert^2\,dx\Bigr] =C_1 M(v(t)).
\end{align*}
This yields \eqref{CQB0}.  Finally, \eqref{CQB0a} follows from \eqref{CQB0} and Lemma \ref{LEM:CQ1}.
\end{proof}

We now turn to the proof of Theorem \ref{THM:CQ}, which we again divide into two pieces: global existence and unconditional uniqueness.

\subsection{Global existence} \label{SUBSEC:CQ2}
As in Section \ref{SEC:GP}, we view \eqref{CQNLS3}
as a perturbed energy-critical quintic NLS on $\R^3$ with the perturbation $e =\RR(v)$ defined in \eqref{CQNLS4}
and apply Lemma \ref{LEM:Perturb}. The argument is basically the same as that in Section \ref{SEC:GP} with one difference:
As the conserved Hamiltonian $E(v)$ no longer controls the $\dot{H}^1_x$-norm of the solution, we have to rely instead on Lemma~\ref{LEM:CQ3}.
Notice that the bound is not uniform in time; however, by Lemma~\ref{LEM:CQ3},
\begin{equation} \label{CQQ0}
\|v(t)\|_{\dot{H}^1(\R^3)} \leq C\big(E(v_0), \lVert\Re(v_0)\rVert_{L^2(\R^3)}, \tau\big) = : N(v_0, \tau),
\end{equation}
for any $0 \leq t \leq \tau$ for which the solution $v$ exists.

Let $v_0$ be such that $1+v_0 \in \mathcal{E}_\text{CQ}(\R^3)$ and fix $\tau>0$.  Our goal is to construct a solution $v$ to \eqref{CQNLS3} on $[0, \tau]$.
It suffices to prove that there exists $T = T(N(v_0, \tau)) > 0$ such that
\begin{equation}
\|v\|_{\dot{S}^1([T_0, T_0+ T]\times \R^3)} \leq C(N(v_0, \tau))
\label{CQQ2}
\end{equation}
as long as $[T_0, T_0+ T] \subset [0, \tau]$.  This allows us to iterate the local argument and extend the solution $v$ to $[0, \tau]$.
Since the choice of $\tau > 0$ is arbitrary, this yields a global solution $v$ to \eqref{CQNLS2}, which is unique in $\dot S^1_{\loc}$.

Fix $T> 0$ to be chosen later (depending on $N(v_0, \tau)$) and suppose $T_0\geq 0$ is such that $[T_0, T_0+ T] \subset [0, \tau]$.
Let $w$ be the global solution to the energy-critical quintic NLS \eqref{NLS1} with initial data $v(T_0)$. As before, we divide $[T_0, \infty)$ into
$J=J(E(v(T_0)) ,\eta)=J(N(v_0,\tau), \eta)$ many subintervals $I_j=[t_j,t_{j +1}]$ such that
\begin{equation} \label{CQQ1}
\| w \|_{\dot{X}^1(I_j)} = \|\nabla w\|_{L^{10}_tL^{\frac{30}{13}}_x(I_j \times \R^3)}\sim \eta,
\end{equation}
for some small $\eta> 0$ to be chosen later.  We write $[T_0, T_0+T] =\bigcup_{j=0}^{J'}\big([T_0,T_0+T ]\cap I_j\big) $ for some $J' \leq J$,
where $[T_0,T_0+T]\cap I_j\neq \emptyset$ for $0\leq j\leq J'$.

By the Strichartz estimate, Sobolev embedding, and \eqref{CQQ1}, we have
\begin{align*}
\big\| e^{i(t-t_j)\Dl} w(t_j)\big\|_{\dot{X}^1(I_j)}
&\leq \|w\|_{\dot{X}^1(I_j)} +C \|w^4\nabla w\|_{L^2_{t}L^\frac{6}{5}_x(I_j\times\R^3)}\\
& \leq \eta + C\|\nabla w\|_{L^{10}_{t} L^\frac{30}{13}_x (I_j \times\R^3)}\|w\|_{L^{10}_{t, x}(I_j\times \R^3)}^4\\
& \leq \eta + C\|\nabla w\|_{L^{10}_tL^\frac{30}{13}_x(I_j\times \R^3)}^5\\
& \leq \eta+C\eta^5,
\end{align*}
where $C$ is an absolute constant. Therefore, if $\eta$ is small enough, we obtain
\begin{align}\label{CQQ3}
\big\| e^{i(t-t_j)\Dl} w(t_j)\big\|_{\dot{X}^1(I_j)} \leq 2\eta.
\end{align}

Arguing similarly and using \eqref{CQQ0} and \eqref{CQQ3}, we obtain
\begin{align}
\|v\|_{\dot{X}^1(I_0)}
& \leq \| e^{it\Delta}v(T_0)\|_{\dot{X}^1(I_0)}+ C\|v^4\nabla v\|_{L^2_t L^\frac{6}{5}_x (I_0\times \R^3)}
+C\|v^3\nabla v\|_{L^\frac{20}{13}_t L^\frac{30}{22}_x (I_0\times \R^3)} \notag \\
& \hphantom{XXX}+ C\|v^2\nabla v\|_{L^\frac{5}{4}_t L^\frac{30}{19}_x (I_0\times \R^3)}
+C\|v\nabla v\|_{L^\frac{20}{19}_t L^\frac{30}{16}_x (I_0\times \R^3)} + C\|\nabla v\|_{L^1_t L^2_x (I_0\times \R^3)} \notag \\
& \leq \| e^{it\Delta}v(T_0)\|_{\dot{X}^1(I_0)}
 +C \sum_{\ell = 2}^5 |I_0|^{\frac{5-\ell}{4}}\| v\|_{L^{10}_{t, x} (I_0\times \R^3)}^{\ell-1} \|\nabla v\|_{L^{10}_tL^{\frac{30}{13}}_x(I_0 \times \R^3)} \notag \\
& \hphantom{XXX} +C|I_0|\|\nabla v\|_{L^\infty_tL^2_x(I_0 \times \R^3)} \notag \\
& \leq  2\eta +C \sum_{\ell = 2}^5T^{\frac{5-\ell}{4}} \|v\|^{\ell} _{\dot{X}^1(I_0)} + CTN(v_0, \tau).
\label{CQQ3c}
\end{align}
Choosing  $\eta \ll 1$ and  $T\ll \min\{1,\frac{\eta}{N(v_0, \tau)}\}$, it follows from a standard continuity argument that
 \begin{equation} \label{CQQ4}
  \|v\|_{\dot{X}^1(I_0)}\leq 3\eta.
\end{equation}

Then, by Sobolev embedding, $\|v\|_{L^{10}_{t, x} (I_0 \times \R^3)} \leq C\eta$ and so condition \eqref{P1} in Lemma \ref{LEM:Perturb} is satisfied with $L=C\eta$.
As in Section~\ref{SEC:GP}, $\eta$ is a small absolute constant.

By \eqref{CQQ0}, condition \eqref{P2} is satisfied with $E_0=N(v_0, \tau )$.  With $t_0 = T_0$, conditions \eqref{P3} and \eqref{P4} are automatically satisfied
(with $E' = 0$ for \eqref{P3}).

Let us now verify condition \eqref{P5}.  Recall that the perturbation $e$ in Lemma~\ref{LEM:Perturb} is $e= \RR(v)$ defined in \eqref{CQNLS4}.
Proceeding as in \eqref{CQQ3c} and using \eqref{CQQ4}, we have
\begin{align}
\|\nabla e\|_{\dot{N}^0(I_0)}
&\leq   C \sum_{\ell = 2}^4T^{\frac{5-\ell}{4}} \|v\|^{\ell} _{\dot{X}^1(I_0)} + CTN(v_0, \tau)\notag \\
& \leq C\sum_{\ell = 2}^4 T^{\frac{5-\ell}{4}} \eta^\ell + CTN(v_0, \tau).
\label{CQQ4a}
\end{align}
For any $\eps>0$, we may choose $T= T(N(v_0, \tau),  \eps)$ sufficiently small so that
$$
\|\nabla e\|_{\dot{N}^0(I_0)}\leq \eps.
$$
Hence, condition \eqref{P5} holds provided $\eps\leq \eps_0(N(v_0, \tau ))$ where $\eps_0(N(v_0, \tau ))$ is dictated by Lemma~\ref{LEM:Perturb}.

Therefore, all hypotheses of Lemma~\ref{LEM:Perturb} are satisfied on the interval $I_0$, provided that $T = T(N(v_0, \tau), \eps)$ is sufficiently small.
Thus we obtain
\begin{align}
\|w-v\|_{\dot{S}^1(I_0\times\R^3)} & \leq C\bigl(N(v_0, \tau )\bigr)\eps. \label{CQQ5b}
\end{align}

The rest of the argument follows exactly as in Section \ref{SEC:GP}; we omit the details.
At the end of the day,  we obtain
\begin{equation*}
\|v\|_{\dot{S}^1([T_0, T_0+T]\times \R^3)} \leq C\bigl(N(v_0, \tau )\bigr)= C\big(E(v_0), \lVert\Re(v_0)\rVert_{L^2(\R^3)}, \tau\big),
\end{equation*}
which settles \eqref{CQQ2} and thus the global existence question.

\subsection{Unconditional uniqueness}

We turn now to showing that the global solutions constructed above are unique among those that are continuous (in time) with values in the energy space.
The proof is similar to that for the Gross--Pitaevski equation given in Section~\ref{SEC:GP}.  Let $v_0$ be such that $1+v_0\in \mathcal{E}_\text{CQ}(\R^3)$
and let $v$ be the global solution to \eqref{CQNLS3} constructed above.  In particular, $v\in \dot S^1(I)$ for any compact time interval $I$.

Let $\tilde v:[0, \tau]\times\R^3\to \C$ be a second solution to \eqref{CQNLS3}
with the same initial condition such that $1+ \tilde v \in C_t ([0, \tau];\mathcal{E}_\text{CQ}(\R^3))$ and write
$\omega:=v-\tilde v$.  Arguing as in the previous section and shrinking $\tau$ if necessary, we may assume
\begin{align}\label{all small}
\|v\|_{L_{t,x}^{10}([0, \tau]\times\R^3)}&+ \|\omega\|_{L_t^\infty L^6_x([0, \tau]\times\R^3)} +\|\omega\|_{L_t^\infty L^4_x([0, \tau]\times\R^3)} \notag\\
&+ \lVert\Re(\omega)\rVert_{L_t^\infty H^1_x([0, \tau]\times\R^3)}+ \lVert\Im(\omega)\rVert_{L_t^\infty \dot H^1_x([0, \tau]\times\R^3)} \leq \eta
\end{align}
for a small $\eta>0$ to be chosen shortly.
In particular, we have that $\omega\in L^2_tL_x^6([0, \tau]\times\R^3) \cap L^{\frac{8}{3}}_tL_x^4([0, \tau]\times\R^3)$.

Writing
\begin{align*}
\bigl[|v|^4v +\RR(v) \bigr]- \bigl[|\tilde v|^4\tilde v + \RR(\tilde v)\bigr]= O\bigl(|\omega|^5 + |\omega||v|^4  + |\omega|^2 + |\omega||v| + \lvert\Re(\omega)|\bigr)
\end{align*}
and using the Strichartz inequality together with \eqref{all small} and H\"older, we estimate
\begin{align*}
&\|\omega\|_{L_t^2L_x^6} + \|\omega\|_{L_t^{\frac{8}{3}}L_x^4}
+ \lVert\Re(\omega)\rVert_{L_t^\infty L_x^2}\\
&\lesssim \|\omega^5\|_{L_t^2L_x^{\frac{6}{5}}} + \|\omega v^4\|_{L_t^{\frac{10}{9}}L_x^{\frac{30}{17}}} + \|\omega^2\|_{L_t^1L_x^2} +\|\omega v\|_{L_t^1L_x^2}
+\lVert\Re(\omega)\rVert_{L_t^1L_x^2}\\
&\lesssim \|\omega\|_{L_t^2L_x^6} \Bigl[\|\omega\|_{L_t^\infty L^6_x}^4 + \|v\|_{L_{t,x}^{10}}^4\bigr]
+ \tau^{\frac{5}{8}} \|\omega\|_{L_t^{\frac{8}{3}}L_x^4} \Bigl[\|\omega\|_{L_t^\infty L^4_x}+\|v\|_{L_t^\infty L^4_x}\Bigr] +\tau \lVert\Re(\omega)\rVert_{L_t^\infty L_x^2}\\
&\lesssim \eta^4\|\omega\|_{L_t^2L_x^6} +\tau^{\frac{5}{8}} \|\omega\|_{L_t^{\frac{8}{3}}L_x^4}+ \tau\lVert\Re(\omega)\rVert_{L_t^\infty L_x^2}.
\end{align*}
All spacetime norms in the display above are taken on $[0, \tau]\times\R^3$.  Taking $\eta$ sufficiently small and shrinking $\tau$ further if necessary, we obtain
$$
\|\omega\|_{L_t^2L_x^6([0, \tau]\times\R^3)} +\|\omega\|_{L_t^{\frac{8}{3}}L_x^4([0, \tau]\times\R^3)}+  \lVert\Re(\omega)\rVert_{L_t^\infty L_x^2([0, \tau]\times\R^3)}=0,
$$
which proves $v=\tilde v$ almost everywhere on $[0, \tau]\times\R^3$.

This completes the proof of unconditional uniqueness and thus the proof of Theorem~\ref{THM:CQ}.

\begin{ackno} \rm
We are grateful to Professor P.~G\'erard for bringing this problem to our attention.  R.~K.~was partially supported by NSF grant DMS-1001531.
T.~O.~acknowledges support from an AMS--Simons Travel Grant.  M.~V.~was partially supported by the
Sloan Foundation and NSF grant DMS-0901166.
\end{ackno}

\end{document}